\documentclass{lmcs}
\pdfoutput=1

\usepackage{lastpage}
\lmcsdoi{18}{1}{18}
\lmcsheading{}{\pageref{LastPage}}{}{}%
{Jul.~13,~2021}{Jan.~19,~2022}{}

\usepackage[utf8]{inputenc}
\usepackage{amsmath,amsthm,amssymb}

\usepackage[all]{xy}
\usepackage{tikz-cd}

\newcommand{\typef}{\mathbf}
\newcommand{\hset}{\typef{hSet}}
\newcommand{\types}{\typef{Type}}
\newcommand{\termf}{\mathtt}
\newcommand{\inl}{\termf{inl}}
\newcommand{\inr}{\termf{inr}}
\newcommand{\base}{\termf{base}}
\newcommand{\fgloop}{\termf{loop}}

\newcommand{\coe}{\termf{idtoeqv}}

\newcommand{\ZZ}{\mathbb{Z}}

\begin{document}
\title[On the Nielsen--Schreier Theorem in Homotopy Type
  Theory]{On the Nielsen--Schreier Theorem in Homotopy Type
  Theory}

\author[A.W. Swan]{Andrew W Swan}
\address{Department of Philosophy \\ Carnegie Mellon University \\
  Pittsburgh}
\email{\texttt{andrewsw@andrew.cmu.edu}}
\thanks{I gratefully acknowledge the support of the Air Force Office of
Scientific Research through MURI grant FA9550-15-1-0053. Any opinions,
findings and conclusions or recommendations expressed in this material
are those of the authors and do not necessarily reflect the views of
the AFOSR.}

\begin{abstract}
  We give a formulation of the Nielsen--Schreier theorem (subgroups of
  free groups are free) in homotopy type theory using the presentation
  of groups as pointed connected 1-truncated types. We show the
  special case of finite index subgroups holds constructively and the
  full theorem follows from the axiom of choice. We give an example of
  a boolean $\infty$-topos where our formulation of the theorem does
  not hold and show a stronger ``untruncated'' version of the theorem
  is provably false in homotopy type theory.
\end{abstract}

\maketitle

\nocite{schreier}

\nocite{higginscatandgpd} 
\nocite{hottbook}

\section{Introduction}
\label{sec:introduction}

The statement of the Nielsen--Schreier theorem sounds very simple at
first: subgroups of free groups are themselves free. However direct
proofs are known to be surprisingly intricate and difficult. This was
the case for the original proofs by Nielsen \cite{nielsen}, for
finitely generated free groups, and Schreier \cite{schreier},
generalising to all free groups.

However, later on much clearer proofs were developed based on ideas
from topology, the
first by Baer and Levi \cite{baerlevi}.\footnote{There was also a
  slightly earlier proof by Chevalley and Herbrand \cite{chevalleyherbrand}
  along similar lines using Riemann surfaces.}
The idea essentially is that
free groups are precisely the fundamental groups of bouquets of
circles. Any subgroup is then the fundamental group of a covering
space of a bouquet of circles. However, any covering space is
homotopic to the geometric realisation of a graph, so the problem is
reduced to showing that the fundamental groups of graphs are free
groups. This is proved by constructing a spanning tree of the graph,
which is then contracted down to point, leaving the remaining edges
outside the spanning tree as edges from that point to itself, showing
that the graph is homotopy equivalent to a bouquet of circles.

This use of ideas from topology makes the Nielsen--Schreier theorem a
natural candidate for formalisation in homotopy type theory
\cite{hottbook}.  In homotopy type theory we can study spaces from a
synthetic point of view, allowing us to use much simpler definitions
that are easier to deal with in formalisations, while still being
guided by the same topological intuitions. We will give a new proof of
the Nielsen--Schreier theorem making essential use of types with non
trivial higher structure, higher inductive types and univalence,
providing an interesting example of a proof using these ideas of a
result that is often stated in a concrete purely algebraic way.

L\"{a}uchli showed in \cite{lauchli} that the use of some form of the
axiom of choice is strictly necessary for the Nielsen--Schreier
theorem, by proving that it fails in a Fraenkel-Mostowski model of
$\mathbf{ZFA}$.\footnote{Later on Howard \cite{howardnsac} and
  Kleppmann \cite{kleppmann15} gave stronger results that further
  clarify the precise relationship between the Nielsen--Schreier
  theorem and the axiom of choice.}
We will show how this result manifests in homotopy type theory by
giving an example of a boolean $\infty$-topos where it is false,
together with a stronger ``untruncated'' version that is provably
false in homotopy type theory.

We will assume throughout that the reader is familiar with standard
ideas in homotopy type theory such as transport, hlevel, higher
inductive types including truncation, connectedness and
univalence. See \cite{hottbook} for all of these concepts.

\subsection*{Agda Formalisation}
\label{sec:agda-formalisation}

The finite index case (Theorem~\ref{thm:nsfiniteindex}) has been
verified electronically using the Agda proof assistant and the
HoTT-Agda library \cite{hottagda}. It is available at
\url{https://github.com/awswan/nielsenschreier-hott}.

\subsection*{Acknowledgements}
\label{sec:acknowledgements}

I am grateful to Mathieu Anel, Steve Awodey, Thierry Coquand and Jonas
Frey for helpful discussion and suggestions. I would also like to
thank the anonymous referees for their helpful suggestions.

\section{Group Theory and Higher Group Theory in HoTT}
\label{sec:group-theory-higher}

In homotopy type theory we can give an alternative definition of group
based on the idea of thinking of a group as the fundamental group of
some space:

\begin{defi}
  A \emph{group} is a pointed type $(BG, \base)$ such that $BG$ is
  $1$-truncated and connected. A group homomorphism $(BG, \base_G) \to
  (BH, \base_H)$ is a function $f : BG \to BH$ together with a proof
  of $f(\base_G) = \base_H$.
\end{defi}

One can show that there is an exact correspondence between groups in
the above sense, and the more usual definition of group as a set with
binary operation satisfying axioms. Given a group $(BG, \base)$ as
above, we define $G$ to be the loop space $\Omega(BG, \base) := \base =
\base$.
This has a
binary operation $\cdot$ given by composition of paths. As shown by
Licata and Finster \cite{licatafinster},
every group in the usual sense is isomorphic,
and so by univalence, equal to such a
loop space using the higher inductive types of Eilenberg-MacLane
spaces.

As shown by Buchholtz, Van Doorn and Rijke
\cite{buchholtzvdoornrijke}, one of the advantages of this approach is
that it easily generalises to higher dimensions. In particular we can
define $\infty$-groups using a slightly simpler definition:
\begin{defi}
  An \emph{$\infty$-group} is a pointed type $(BG, \base)$ such that $BG$ is
  connected.
\end{defi}

We can understand subgroups in this setting using the notion of
\emph{covering space} ~\cite{favoniacoveringsps}.

\begin{defi}
  Let $(BG, \base)$ be a group. A \emph{covering space} on $BG$ is a function
  $BG \to \hset$.

  A \emph{pointed covering space} is a covering space $X : BG \to
  \hset$ together with an element of $X(\base)$.

  We say a covering space $X : BG \to \hset$ is \emph{connected} if the
  total space $\sum_{z : BG} X z$ is connected.

  We say a covering space $X : B G \to \hset$ has \emph{index} $I$ if
  there merely exists an equivalence between $X(\base)$ and $I$. In
  particular, we say it has \emph{finite index} if $X(\base)$ is merely
  equivalent to an initial segment of $\mathbb{N}$.
\end{defi}

Pointed connected covering spaces on $B G$ correspond precisely to
subgroups of $G$ \cite[Theorem 7.1(3)]{buchholtzvdoornrijke}. We will
therefore sometimes refer to them simply as subgroups.

Free groups in this setting were studied by Kraus and Altenkirch
\cite{krausfreehighergps}. We recall some of their results below.

We first define the free higher group $B F_A^\infty$ as the higher
inductive type generated by the following constructors:
\begin{enumerate}
\item $B F_A^\infty$ contains a point $\base$.
\item For each $a : A$ we add a path $\fgloop(a)$ from $\base$ to
  $\base$.
\end{enumerate}

If $A$ is a set with decidable equality then $B F_A$ is $1$-truncated,
and so a group, as defined above. It is currently an open problem
whether this can be proved constructively for sets $A$ in general. We
therefore define the free group on $A$ to be the $1$-truncation
$B F_A := \| B F_A^\infty \|_1$.

\newcommand{\ap}{\termf{ap}}

We can equivalently characterise the free group on $A$ using
any of the following descriptions.
\begin{enumerate}
\item The $1$-truncation of the wedge product of $A$ copies of
  $(\mathbb{S}^1, \base)$.
\item The coproduct of $A$ copies of the group $\mathbb{Z}
  :=(\mathbb{S}^1, \base)$
  in the category of groups and group homomorphisms.
\item The $1$-truncation of the coequalizer of the graph
  $A \rightrightarrows 1$.
\item The unique group $(B F_A, \base)$ equipped with a map
  $\fgloop : A \to \Omega(B F_A, \base_{F_A})$ satisfying the universal property
  that for any group $(BG, \base_G)$ and any map
  $g: A \to \Omega(BG, \base_G)$, there is a unique homomorphism
  $h : (B F_A, \base_{F_A}) \to (BG, \base_G)$ such that for all $a : A$,
  $\ap_h(\fgloop(a)) = g(a)$.
\end{enumerate}

\section{Coequalizers in HoTT}
\label{sec:hits}

In this section we review the definition of coequalizers in homotopy
type theory and show some useful lemmas. We will omit some of the
formal details. See the Agda formalisation for complete
proofs.\footnote{These lemmas are in the directory {\tt main/Coequalizers}.}

\newcommand{\edge}{\termf{edge}}

\begin{defi}
  A \emph{graph} consists of two types $V$ and $E$ together with two
  maps $\pi_0, \pi_1 : E \rightrightarrows V$. We will refer to
  elements of $V$ as \emph{vertices} and elements of $E$ as
  \emph{edges}.
\end{defi}

\begin{defi}
  Suppose we are given a graph $\pi_0, \pi_1 : E \rightrightarrows
  V$. The \emph{coequalizer} of $(V, E, \pi_0, \pi_1)$, denoted
  $V / E$ when $\pi_0$ and $\pi_1$ are clear from the context, is the
  higher inductive type generated by the following constructors.
  \begin{enumerate}
  \item For every $v : V$, $V / E$ contains a point $[ v ]$.
  \item For every $e : E$ there is a path
    $\edge : [ \pi_0(e) ] = [ \pi_1(e) ]$ in $V / E$.
  \end{enumerate}
\end{defi}

We will use the following three key lemmas about coequalizers in the
proof.

In the first lemma we are given a graph where the type of edges is a
coproduct of two types $E_0$ and $E_1$. We show that the coequalizer
$V / E_0 + E_1$ can be computed in two steps, first quotienting by
$E_0$, and then by $E_1$. We can visualise this as follows. Suppose we
are given a topological space $V$, and produce a new space by gluing
on a set of intervals indexed by $E$. Then we obtain the same space by
first gluing on half of the intervals, and then separately gluing on
the other half.

\begin{lem}
  \label{lem:coeqcoprod}
  Suppose we are given types $E_0, E_1, V$ together with a pair of
  maps $\pi_0, \pi_1 : E_0 + E_1 \rightrightarrows V$. By composing
  with the coproduct inclusion we get a diagram $E_0 \rightrightarrows V$,
  and so a type $V / E_0$ given by coequalizer. We then obtain
  a pair of maps $E_1 \rightrightarrows V / E_0$ by composing with the
  other coproduct inclusion and the map $[-] : V \to V / E_0$. We then
  have the equivalence below.
  \begin{equation*}
    V / (E_0 + E_1) \simeq (V / E_0) / E_1
  \end{equation*}
\end{lem}

\begin{proof}
  Functions in both directions
  $f : V / (E_0 + E_1) \to (V / E_0) / E_1$ and
  $g : (V / E_0) / E_1 \to V / (E_0 + E_1)$ can be constructed by
  recursion on coequalizers. One can then show $f \circ g \sim 1$ and
  $g \circ f \sim 1$ by induction on the construction of the
  coequalizers.
\end{proof}

We visualise the second lemma as follows. We are given a space $X$
together with a point $x \in X$. We extend $X'$ to a larger space by
adding an extra path $e$ with one endpoint attached at $x$. Then $X'$
is homotopy equivalent to $X$, since we can contract the new path $e$
down to the point $x$. The analogous construction in algebraic
topology is sometimes known as ``growing a whisker.''
\begin{lem}
  \label{lem:trivextn}
  Let $X$ be a type with an element $x : X$. Define two maps
  $1 \rightrightarrows X + 1$ corresponding to the two elements
  $\inl(x)$ and $\inr(\ast)$ of $X + 1$. Then the canonical map to the
  coequalizer $f : X \hookrightarrow X + 1 \rightarrow X + 1 / 1$ is
  an equivalence.
\end{lem}

\begin{proof}
  We can construct an inverse $g : (X + 1) / 1 \,\to\, X$ by recursion on
  the definition of coequalizer. Note that $g \circ f$ is
  definitionally equal to $1$, and we can show $f \circ g \sim 1$ by
  induction on the definition of $(X + 1) / 1$.
\end{proof}

\begin{lem}
  \label{lem:coeqpb}
  ``Coequalizers are stable under pullback.'' Suppose we are given a
  graph $E \rightrightarrows V$ and a family of types
  $X : V / E \to \types$. Define $E' := \sum_{e : E} X([\pi_0(e)])$
  and $V' := \sum_{v : V} X([v])$. Define $\pi_0', \pi_1' : E' \to V'$
  by $\pi_0'(e, x) := (\pi_0(e), x)$ and $\pi_1'(e, x) :=
  \edge(e)_\ast(x)$
  Then $\sum_{z : V / E} X(z) \simeq V' / E'$.
\end{lem}

\begin{proof}
  This can be seen as a special case of the flattening lemma for
  coequalizers \cite[Lemma 6.12.2]{hottbook}.  However, for
  completeness we will give a direct proof.
  
  We first define a map $f : \sum_{z : V / E} X(z) \to V' /
  E'$. Equivalently we can define a dependent function $f' : \prod_{z
    : V / E} \, (X(z) \to V' / E')$. We define $f'$ using the
  elimination principle of $V / E$. Given $v : V$, we define
  $f'([v])(x)$ to be $[(v, x)]$. Given $e : E$ we need to define a path
  as below.
  \begin{equation}
    \label{eq:1}
    \edge(e)_\ast(f'([\pi_0(e)])) =
    f'(\pi_1(e))
  \end{equation}
  However, by path induction we can show that for all paths $p : z =
  z'$ in $V / E$ and all $h : X(z) \to V' / E'$ that $p_\ast (h)(x)=
  h(p^{-1}_\ast(x))$.
  
  By applying the above with $h = f'(\pi_0(e))$ and $p = \edge(e)$,
  and function extensionality, we can deduce \eqref{eq:1} by finding
  for each $e : E$ and each $x : X(\pi_1(e))$ a path of the type
  below.
  \begin{equation*}
    f'([\pi_0(e)])(\edge(e)^{-1}_\ast(x)) = f'([\pi_1(e)])(x)
  \end{equation*}
  By definition, it suffices to find a path
  $(\pi_0(e), \edge(e)^{-1}_\ast(x)) = (\pi_1(e), x)$ in
  $V'$. However, by the characterisation of identity types for
  $\sum$-types, this is the same as a path $q : \pi_0(e) = \pi_1(e)$
  together with a path as below.
  \begin{equation*}
    q_\ast( \edge(e)^{-1}_\ast(x)) = x 
  \end{equation*}
  We can of course take $q := \edge(e)$.

  We will define $g : V' / E' \to \sum_{z : V / E} X(z)$ by recursion on the
  construction of $V' / E'$. We define $g([(v, x)]) := ([v],
  x)$. Given $(e, x) : E'$, we have an evident path
  $([\pi_0(e)], x) = ([\pi_1(e)], \edge(e)_\ast(x))$, which
  gives us the well defined function $g$.

  Finally, one can verify $f \circ g \sim 1$ by induction on the
  definition of $V' / E'$, and $g \circ f \sim 1$ by induction on the
  definition of $V / E$.
\end{proof}

We will often also implicitly use the following lemma.
\begin{lem}
  Suppose we are given graphs $\tau_0, \tau_1 : D \rightrightarrows U$
  and $\pi_0, \pi_1 : E \rightrightarrows V$ together with
  equivalences $D \simeq E$ and $U \simeq V$ commuting with the
  endpoint maps. Then we have an equivalence $U / D \simeq V / E$.
\end{lem}

\begin{proof}
  For convenience we will assume that all types involved lie at the
  same universe level.\footnote{For the more general statement where the types can have different universe levels, we need an additional lemma that coequalizers are preserved by ``lifting'' to higher universe levels. See the file {\tt main/Coequalizers/PreserveEquivalence.agda} in the formalisation for details.} By univalence we may assume that $D = E$ and
  $U = V$, and that the equivalences are given by transport along
  these paths. Hence it suffices to show that for all types
  $D, E, U, V$, all paths $p : D = E$ and $q : U = V$, all maps
  $\tau_0, \tau_1 : D \rightrightarrows U$ and
  $\pi_0, \pi_1 : E \rightrightarrows V$ and finally all proofs that
  $q_\ast \circ \tau_i = \pi_i \circ p_\ast$ we have
  an equivalence $U / D \simeq V / E$. However, the preceding
  statement can be proved by iterated path induction.
\end{proof}

\section{Spanning Trees in HoTT}
\label{sec:some-basic-graph}

We think of the coequalizer of a graph $E \rightrightarrows V$ as its
\emph{geometric realisation}, the topological space that has a point
for each vertex $v : V$, and a path from $\pi_0(e)$ to $\pi_1(e)$ for
each edge $e : E$. Note that a graph is connected if and only if its
geometric realisation is connected. Similarly a graph is a tree if and
only if its geometric realisation is contractible, or equivalently if
it is both connected and $0$-truncated (contains no non trivial
cycles). We will take this topological point of view as the definition
of connected and tree.

\begin{defi}
  Let $E \rightrightarrows V$ be a graph. We say the graph is
  \emph{connected} if $E / V$ is a connected type, and we say the
  graph is a \emph{tree} if $E / V$ is contractible.
\end{defi}

\begin{defi}
  Let $\pi_0, \pi_1 : E \rightrightarrows V$ be a graph. A
  \emph{subgraph} is a graph $D \rightrightarrows U$ together with
  embeddings $h : D \hookrightarrow E$ and $k : U \hookrightarrow V$
  such that the following squares commute for $i = 0, 1$:
  \begin{equation*}
    \begin{tikzcd}
      D \ar[d, "\tau_i"] \ar[r, hook, "h"] & E \ar[d, "\pi_i"] \\
      U \ar[r, hook, "k"] & V
    \end{tikzcd}
  \end{equation*}
\end{defi}

\begin{defi}
  Let $E \rightrightarrows V$ be a graph. A \emph{spanning tree} is a
  subgraph $D \rightrightarrows U$ such that $D \rightrightarrows U$
  is a tree, the embedding $U \hookrightarrow V$ is an equivalence and
  the embedding $D \hookrightarrow E$ has decidable
  image.\footnote{It is not clear if decidable image should be part of
    the definition of spanning tree for every application. However, the
    main use of spanning trees for this paper will be
    Lemma~\ref{lem:spanningtreetofreegp} below, where it does
    play a critical role.}
\end{defi}

We will give two lemmas on the existence of spanning trees. Both
will use Lemma ~\ref{lem:extendsubgraph}, which in turn uses the lemma
below.

\begin{lem}
  \label{lem:crossingedge}
  Suppose we are given a connected graph $E \rightrightarrows V$, and
  that $V$ decomposes as a coproduct $V \simeq V_0 + V_1$. Suppose
  further that we are given an element of each component $v_0 : V_0$ and
  $v_1 : V_1$. Then there merely exists an edge $e : E$ such that either
  $\pi_0 (e) \in V_0$ and $\pi_1(e) \in V_1$, or $\pi_0(e) \in V_1$ and
  $\pi_1(e) \in V_0$.
\end{lem}

\begin{proof}
  We first define a family of propositions $P : V / E \to \types$. We
  wish $P$ to satisfy the following. For $v : V_0$,
  $P([\inl(v)]) = 1$, and for $v : V_1$ we require the equation
  below.
  \begin{equation*}
    P([\inr(v)]) = \left\| \sum_{e : E} \, (\pi_0(e) \in V_0 \,\times\, \pi_1(e)
      \in V_1) \;+\; (\pi_0(e) \in V_1 \,\times\, \pi_1(e) \in V_0)
    \right\|
  \end{equation*}
  To show such a $P$ exists, we note that the requirements above
  precisely define the action on points of a recursive definition on
  $V / E$. Hence to get a well defined function it suffices to define
  an action on paths. That is, we need equalities
  $P([\pi_0(e)]) = P([\pi_1(e)])$ for $e : E$. By propositional
  extensionality we just need to show $P([\pi_0(e)])$ and
  $P([\pi_1(e)])$ are logically equivalent. However, this is
  straightforward by considering the $4$ cases depending on whether
  $\pi_i(e) \in V_0$ or $\pi_i(e) \in V_1$ for $i = 0,1$: if
  $\pi_0(e)$ and $\pi_1(e)$ lie in the same component of the
  coproduct, then $P([\pi_0(e)])$ and $P([\pi_1(e)])$ are the same by
  definition, and if they lie in different components, then $P([v])$
  is true for all $v : V$.  We now construct a map from
  $[\inl(v_0)] = z$ to $P(z)$ for each $z : V / E$. By based path
  induction it suffices to construct an element of $P(\inl(v_0))$, but
  this was defined to be $1$, so is trivial. By connectedness,
  there merely exists an element of $[\inl(v_0)] = [\inl(v_1)]$, and
  so $P(\inl(v_1))$ is inhabited and the lemma follows.
\end{proof}

To further illustrate the proof of Lemma~\ref{lem:crossingedge} we
give an alternative non constructive proof using the same
family of types $P$.

Using the law of excluded we can assume there is no edge with endpoints in
different components of $V$ and derive a contradiction. Under this
assumption, our requirement on $P$ is that it is $1$ on $V_0$ and $0$
on $V_1$. In other words we have a $2$-colouring of vertices and want
to extend it to a $2$-colouring on the whole graph. We can do this by
the assumption, since any edge has the same colour on both its
endpoints, so we can take the whole edge to be that colour. This now
contradicts connectedness, since we have a surjection from the graph
to $2$. Topologically, we can think of this as a continuous surjection
from a connected space to the discrete space $2$, which is not possible.

\begin{lem}
  \label{lem:extendsubgraph}
  Let $E \rightrightarrows V$ be a connected graph where $V$ has
  decidable equality and $E$ is a set, together with a subgraph
  $D \rightrightarrows U$ such that the inclusion $U \hookrightarrow
  V$ has decidable image.

  Suppose further that we are given elements
  $u \in U$ and $v \in V \setminus U$. Then there merely exists a
  larger subgraph whose type of vertices is $U + 1$ and whose type of
  edges is $D + 1$, such that the canonical map $U / D \to (U + 1) /
  (D + 1)$ is an equivalence, as illustrated below:
  \begin{equation*}
    \begin{tikzcd}
      D \ar[d, shift right] \ar[d, shift left] \ar[r, hook] & D + 1
      \ar[d, shift right] \ar[d, shift left] \ar[r, hook] &
      E \ar[d, shift right] \ar[d, shift left] \\
      U \ar[r, hook] \ar[d] & U + 1 \ar[r, hook] \ar[d] & V \ar[d] \\
      U / D \ar[r, "\simeq"] & (U + 1) / (D + 1) \ar[r] & V / E
    \end{tikzcd}
  \end{equation*}
\end{lem}

\begin{proof}
  Since the inclusion $U \hookrightarrow V$ is decidable, we can write
  $V$ as the coproduct of $U$ with its complement
  $V \simeq U + V \setminus U$. Applying Lemma~\ref{lem:crossingedge}
  shows there merely exists an edge $e : E$ such that either $\pi_0(e)
  \in U$ and $\pi_1(e) \notin U$ or vice versa. We consider the former
  case, the latter being similar.

  We define the map $U + 1 \hookrightarrow V$ to be the same as the
  map $U \hookrightarrow V$ on the $U$ component and to be equal to
  $\pi_1(e)$ on the $1$ component. Since $\pi_1(e) \notin U$, this is
  an embedding. Similarly we define the map $D + 1 \hookrightarrow E$
  by taking the $1$ component to $e$.  Note that we cannot have
  $e \in D$, since this would imply $\pi_1(e) \in U$, and so the map
  $D + 1 \hookrightarrow E$ is also an embedding.

  We define the endpoint maps $D + 1 \rightrightarrows U + 1$ as
  appropriate to satisfy commutativity conditions.

  Finally, we verify the equivalence by computing as follows.
  \begin{align*}
    (U + 1)/(D + 1) &\simeq ((U + 1) / 1) / D
    & \text{Lemma \ref{lem:coeqcoprod}} \\
                    &\simeq U / D & \text{Lemma \ref{lem:trivextn}} & \qedhere
  \end{align*}
\end{proof}

\begin{lem}
  \label{lem:spanningtreefinite}
  Let $E \rightrightarrows V$ be a connected graph where $V$ is
  finite, say with $|V| = n$ and $E$ is a set with decidable
  equality. Then the graph merely has a spanning tree
  $D \hookrightarrow E \rightrightarrows V$ where $D$ is finite with
  $|D| = n - 1$.
\end{lem}

\begin{proof}
  We show by induction that for $1 \leq k \leq n$ there merely exists a
  subgraph $D \rightrightarrows U$ such that $|U| = k$, $|D| = k - 1$
  and $U / D$ is contractible.

  For $k = 1$ we observe that by the definition of connectedness, $V /
  E$ is merely inhabited, and so $V$ is also merely inhabited. 
  An element $v$ of $V$ defines an embedding $1 \hookrightarrow V$,
  and it is clear that the coequalizer $1 / 0$ is contractible.

  Now suppose we have already defined a suitable subgraph for
  $1 \leq k < n$, say $D \rightrightarrows U$. Since $V$ is finite, it
  in particular has decidable equality. Furthermore, since $|U| = k$
  with $1 \leq k < n$, there exist $u \in U$ and $v \notin U$. Hence
  we can apply Lemma~\ref{lem:extendsubgraph} to show the existence of
  a subgraph of the form $D + 1 \rightrightarrows U + 1$ where
  $(U + 1) / (D + 1) \simeq U / D$. Since $U / D$ is contractible, so is
  $(U + 1) / (D + 1)$, as required.

  Now we apply the above with $k = n$ to get a subgraph
  $D \rightrightarrows U$ where $U / D$ is contractible. Since $U$ and
  $V$ are both finite of the same size $n$, the embedding $U
  \hookrightarrow V$ is an equivalence. Since $E$ has decidable
  equality and $D$ is finite, the embedding $D \hookrightarrow E$ has
  decidable image. Hence this does indeed give a subtree.
\end{proof}

\begin{lem}
  \label{lem:spanningtreezl}
  Let $E \rightrightarrows V$ be a connected graph where $E$ and $V$
  are both sets. Suppose that the axiom of choice holds. Then a
  spanning tree for the graph merely exists.
\end{lem}

\begin{proof}
  Recall that the axiom of choice implies the law of excluded middle
  and Zorn's lemma.
  
  We consider the set of subgraphs of $E \rightrightarrows V$ that are
  trees, ordered by inclusion.

  We verify that the poset is chain complete. If we are given a chain
  of subgraphs $(D_i \rightrightarrows U_i)_{(i : I)}$, take
  $D \rightrightarrows U$ to be the union of all the subgraphs, and
  write $\iota_i$ for the canonical map $U_i / D_i \rightarrow U /
  D$. Fix $u : U$, noting that such a $u$ merely exists since $I$ is
  merely inhabited\footnote{We follow the convention that chains are
    inhabited.} and $U_i/D_i$ is contractible for each $i$. For
  $v : U$, we choose\footnote{This requires an application of the
    axiom of choice to inhabited subsets of $I$.}
  $i : I$ such that $u, v \in U_i$, and take $p$ to
  be the unique path $[u] = [v]$ in $U_i$. This then gives us a path
  $\iota_i(p)$ in $U / D$. Similarly, for each $e : D$, we can choose
  an element $i : I$ such that $e \in D_i$, and from this construct
  a homotopy between the choice of path from $[u]$ to $[\pi_0(e)]$
  composed with $e$ and the choice of path from $[u]$ to
  $[\pi_1(e)]$. Combined with the induction principle for $U / D$,
  this gives us a path from $[u]$ to $z$ for each $z : U / D$, showing
  that $U / D$ is contractible.
  
  Hence the poset has a maximal element $D \rightrightarrows U$ by
  Zorn's lemma. We wish to show every element of $V$ belongs to
  $U$. By the law of excluded middle, it suffices to derive a
  contradiction from the assumption of $v \in V \setminus U$. However,
  by Lemma ~\ref{lem:extendsubgraph} we could obtain a larger tree
  subgraph, contradicting maximality, as required.

  Again using excluded middle, $D$ has a complement in $E$, giving us
  the spanning tree.
\end{proof}

We next see a key lemma that establishes the geometric realisation of
graphs that have spanning trees are equivalent to bouquets of
circles. The way to visualise this is that we contract the spanning
tree down to a single point. This leaves the remaining edges not in
the spanning tree as loops from this single point to itself.
\begin{lem}
  \label{lem:spanningtreetofreegp}
  Suppose that $E \rightrightarrows V$ is a graph with a spanning tree
  $E_0 \hookrightarrow E_0 + E_1 \simeq E$. Then $V/(E_0 + E_1)$ is
  equivalent to the free $\infty$-group on $E_1$.
\end{lem}

\begin{proof}
  \begin{align*}
    V/(E_0 + E_1) &\simeq (V / E_0) / E_1 &\text{Lemma \ref{lem:coeqcoprod}} \\
                  &\simeq 1/E_1 &\text{since $E_0$ is a spanning tree} \\
                  &\simeq BF_{E_1}^\infty & & \qedhere
  \end{align*}
\end{proof}

\section{The Nielsen--Schreier Theorem}
\label{sec:niels-schr-theor}

We now prove two versions of the Nielsen--Schreier theorem in HoTT.
Following the classical proofs, we proceed in two steps. We first show
that every subgroup of a free group is equivalent to the geometric
realisation of a graph. We then use the results of Section
~\ref{sec:some-basic-graph} to deduce that it equivalent to a free
group, under certain conditions.

\begin{lem}
  \label{lem:subinftygpisfgpd}
  ``Every bundle on a free $\infty$-group is the geometric
  realisation of a graph.'' Let $A$ be any type, and
  $(B F_A, \base)$ the free $\infty$-group on $A$ generated by paths
  $\fgloop(a)$ for $a : A$. Let $X : B F_A \to \types$ be any family of types
  over $B F_A$. We define a graph
  $\pi_0, \pi_1 : A \times X(\base) \to X(\base)$ by taking $\pi_0$ to
  be projection, and define $\pi_1(a, x)$ to be $\fgloop(a)_\ast(
  x)$. We then have
  \begin{equation*}
    \sum_{z : B F_A^\infty} X(z) \quad\simeq\quad X(\base)/(A \times X(\base))
  \end{equation*}
\end{lem}

\begin{proof}
  Note that $B F_A$ is equivalent to the coequalizer of a graph
  $A \rightrightarrows 1$ with $\base = [\ast]$ where $\ast$ is the
  unique element of $1$. Hence we can apply Lemma~\ref{lem:coeqpb}
  to express $\sum_{z : B F_A} X(z)$ as a coequalizer $V' /
  E'$. However, we then have the following definitional equalities and
  equivalences.
  \begin{equation*}
    V' \;\equiv\; \sum_{v : 1} X([\ast])
    \;\simeq\; X([\ast])
    \;\equiv\; X(\base)
  \end{equation*}
  \begin{equation*}
    E' \;\equiv\; \sum_{a : A} X([\pi_0(1)])
       \;\equiv\; \sum_{a : A} X([\ast])
       \;\equiv\; \sum_{a : A} X(\base)
       \;\simeq\; A \times X(\base) \qedhere
  \end{equation*}
\end{proof}

In order to derive the truncated version of
Lemma~\ref{lem:subinftygpisfgpd} we first recall the following
flattening lemma for truncation.

\begin{lem}
  \label{lem:truncflatten}
  Suppose we are given a type $Y$ and a family of sets $X : \| Y \|_1 \to \hset$. Then $\sum_{z : \| Y \|_1} X(z) \simeq \| \sum_{y : Y} X(| y |_1) \|_1$.
\end{lem}

\begin{proof}
  Similar to the proof of Lemma~\ref{lem:coeqpb}.
\end{proof}

\begin{lem}
  \label{lem:subgpisfgpd}
  ``Every subgroup of a free group is the geometric realisation of a
  graph.''  Let $A$ be a set, and $(B F_A, \base)$ the free group on
  $A$. Let $X : B F_A \to \hset$
  be a covering space on $B F_A$. Then we then have the following.
  \begin{equation*}
    \sum_{z : B F_A} X(z) \quad\simeq\quad \| X(\base)/(A \times
    X(\base)) \|_1
  \end{equation*}
\end{lem}

\begin{proof}
  We define $X' : B F_A^\infty \to \types$ to be the composition of
  $X$ with the truncation map $| - |_1 : B F_A^\infty \to B F_A$ and
  projection from $\hset$ to $\types$. Lemma
  ~\ref{lem:subinftygpisfgpd} then gives us an equivalence
  $\sum_{z : B F_A^\infty} X'(z) \simeq X'(\base)/(A \times
  X'(\base))$. However, $X'(\base)/(A \times
  X'(\base))$ is definitionally equal to $X(\base)/(A \times
  X(\base))$. By Lemma~\ref{lem:truncflatten} we have that
  $\sum_{z : \| B F_A^\infty \|_1} X(z)$ is equivalent to $\| \sum_{z : B
    F_A^\infty} X'(z) \|_1$. Putting these together gives us the required
  equivalence $\sum_{z : \| B F_A^\infty \|_1} X(z) \simeq
  \| X(\base)/(A \times X(\base))\|_1$.
\end{proof}

We now give two versions of the Nielsen--Schreier theorem that hold in
homotopy type theory. The first is entirely constructive and includes
the Nielsen--Schreier index formula.

\begin{thm}
  \label{thm:nsfiniteindex}
  Suppose that $B F_A$ is the free group on a set $A$ with decidable
  equality. Suppose that $X : B F_A \to \hset$ is a finite index,
  connected covering space. Then $\sum_{z : B F_A} X z$ is merely
  equivalent to the classifying space of a free group.

  Moreover, suppose that $A$ is finite of size $n$ and that $X$ is of
  finite index $m$. Then there merely exists an equivalence,
  \begin{equation*}
    \sum_{z : B F_A} X z \quad\simeq\quad B F_{m(n - 1) + 1}
  \end{equation*}
\end{thm}

\begin{proof}
  By Lemma~\ref{lem:subgpisfgpd} we have
  $\sum_{z : BF_A} X(z) \simeq \| X(\base) / (A \times X(\base))
  \|_1$.\footnote{In fact, since $A$ has decidable equality
    $X(\base) / (A \times X(\base))$ is already $1$-truncated, but we
    will not need that here.} Note that
  $X(\base) / (A \times X(\base))$ is a coequalizer where the vertex
  set $X(\base)$ is finite and the edge set $A \times X(\base)$ has
  decidable equality. Hence we can apply Lemma
  ~\ref{lem:spanningtreefinite} to show a spanning tree exists. Hence
  we can apply Lemma~\ref{lem:spanningtreetofreegp} to show
  $X(\base) / (A \times X(\base))$ is equivalent to a free higher
  group. Truncating gives us an equivalence between
  $X(\base) / (A \times X(\base))$ and a free group.

  Now suppose that $A$ is also finite, with $|A| = n$ and
  $|X(\base)| = m$. Then Lemma ~\ref{lem:spanningtreefinite} in fact
  gives us a spanning tree with $m$ vertices and $m - 1$ edges. In
  particular we can write the edge set $A \times X(\base)$ as a
  coproduct $E_0 + E_1$ where $X(\base) / E_0$ is contractible and
  $|E_0| = m - 1$. Lemma~\ref{lem:spanningtreetofreegp} then tells us
  $X(\base) / (A \times X(\base))$ is equivalent to the free (higher)
  group on $E_1$. However $|E_1| = mn - (m - 1) = m(n - 1) + 1$, as
  required.
\end{proof}

\begin{thm}
  \label{thm:nsclassicallogic}
  Assume the axiom of choice.
  Suppose that $B F_A$ is the free group on a set $A$. Let $X :
  B F_A \to \hset$ be any connected covering space. Then $\sum_{z : B F_A}
  X z$ is merely equivalent to a free group.
\end{thm}

\begin{proof}
  Similar to the proof of Theorem~\ref{thm:nsfiniteindex} using Lemmas
  \ref{lem:subgpisfgpd}, \ref{lem:spanningtreezl} and
  \ref{lem:spanningtreetofreegp}.
\end{proof}

\section{A Boolean $\infty$-Topos where the Theorem does not Hold}
\label{sec:boolean-infty-topos}

We recall that L\"{a}uchli proved the following theorem in
\cite[Section IV]{lauchli}.

\begin{thm}[L\"{a}uchli]
  The Nielsen--Schreier theorem is not provable in $\mathbf{ZFA}$,
  Zermelo-Fraenkel set theory with atoms.
\end{thm}

We will sketch out how to adapt the proof to obtain a model of
homotopy type theory with excluded middle where the Nielsen--Schreier
theorem does not hold. In order to do this we will work in a classical
metatheory and we will switch back to
the classical definition of free groups using reduced words. In the
presence of the law of excluded middle this is equivalent to our
earlier definition by \cite[Section 2.2]{krausfreehighergps}.

We first recall the following lemma from L\"{a}uchli's proof.

\begin{lem}
  \label{lem:lauchlilemma}
  Let $F_A$ be a free group and $C \leq F_A$ the subgroup generated by
  elements of the form $a b a^{-1} b^{-1}$ for $a, b \in A$. Let
  $X \subseteq C$ be a set that freely generates $C$. Then $X$ cannot
  be invariant under any transposition $(a \; b)$ for
  $a \neq b \in A$.
\end{lem}

\begin{proof}
  See \cite[Section IV]{lauchli}.
\end{proof}

\newcommand{\names}{\mathbb{A}}
\newcommand{\perma}{\operatorname{Perm}(\names)}

We will construct our example using the ($1$-)topos of \emph{nominal
  sets}\cite{pittsnomsets}. We first recall the basic definitions. We
fix a countably infinite set $\names$. We
write $\perma$ for the group of finitely supported permutations of
$\names$ (i.e. $\pi : A \stackrel{\sim}{\longrightarrow} A$ such that
$\pi(a) = a$ for all but finitely many $a \in \names$).

If $(X, \cdot)$ is a $\perma$-set, $x \in X$ and $A \subseteq \names$,
we say $A$ is a \emph{support} for $x$ if $\pi \cdot x = x$ whenever
$\pi \in \perma$ satisfies $\pi(a) = a$ for all $a \in A$.

The topos of nominal sets is defined to be the full subcategory
of $\perma$-sets consisting of $(X, \cdot)$ such that every element $x
\in X$ has a finite support.

\begin{thm}
  The Nielsen--Schreier theorem is false in the internal logic of the
  topos of nominal sets.
\end{thm}

\begin{proof}
  First recall that $\names$ can itself be viewed as a nominal set by
  taking the action to be $\pi \cdot a := \pi(a)$. Take $F_\names$ to
  be the internal free group on $\names$. Observe, e.g. by verifying
  the universal property that $F_\names$ is just the external
  definition of $F_\names$ together with the action given by the
  action on $\names$ and the universal property.

  We define $C \leq F_\names$ by externally taking it to be the
  subgroup generated by elements of the form $a b a^{-1} b^{-1}$ for
  $a, b \in \names$. We observe that the action of $\perma$ on
  $F_\names$ restricts to $C$, giving us a subgroup $C \leq F_\names$
  in nominal sets. The object of freely generating subsets,
  $\mathcal{G}$ of $C$ can be explicitly described as the set of
  $X \subseteq C$ such that $X$ has finite support and (externally)
  freely generates $C$, with the obvious action.
  The Nielsen--Schreier theorem implies that
  $\mathcal{G}$ contains some element $X$. Let $A \subseteq \names$ be
  a finite support for $X$. Let $a, b$ be distinct elements of $\names
  \setminus A$. Then $(a \; b) \cdot X = X$, contradicting Lemma
  ~\ref{lem:lauchlilemma}.
\end{proof}

\begin{cor}
  There is a boolean $\infty$-topos where the Nielsen--Schreier theorem
  does not hold.

  The Nielsen--Schreier theorem is not provable in homotopy type
  theory, even with the addition of the law of excluded middle.
\end{cor}

\begin{proof}
  The topos of nominal sets is equivalent to a Grothendieck topos
  referred to as the \emph{Schanuel topos} \cite[Section
  6.3]{pittsnomsets}.

  Lurie showed in \cite[Proposition 6.4.5.7]{luriehtt} that any
  Grothendieck ($1$-)topos is equivalent to the $0$-truncated elements
  of some Grothendieck $\infty$-topos. In particular, if we apply this
  to the topos of nominal sets we obtain a boolean Grothendieck
  $\infty$-topos where Nielsen--Schreier does not hold.

  Shulman proved in \cite{shulmaninftytopunivalence} that homotopy
  type theory can be interpreted in any Grothendieck
  $\infty$-topos. We thereby obtain a model of HoTT with the law of
  excluded middle where Nielsen--Schreier does not hold.
\end{proof}

\begin{rem}
  As an alternative to the non constructive methods of Lurie and
  Shulman, it may also be possible to use a cubical sheaf model, as
  developed by Coquand, Ruch and Sattler \cite{coquandruchsattler},
  but we leave a proof for future work.
\end{rem}

\section{The Untruncated Nielsen--Schreier Theorem is False}
\label{sec:untr-niels-schr}

In this section we will again work with the more usual definition of group
as set with a binary operation satisfying the well known axioms. We
will write the free group as $F_A$, which is the loop space of $B F_A$
at $\base$. We recall the observation of Kraus and Altenkirch
\cite[Section 2.2]{krausfreehighergps} that when $A$ has decidable
equality we can use the classical description of $F_A$ as the set of
reduced words with multiplication given by concatenation followed by
reduction. In particular we note that when $A$ has decidable equality
so does $F_A$.

In Theorems \ref{thm:nsfiniteindex} and \ref{thm:nsclassicallogic} we
were careful to state that the equivalences merely exist, to emphasise
that formally we are only constructing an element of the truncation
$\| \sum_{z : B F_A} X z \;\simeq\; B F_B \|$, not the type $\sum_{z :
  B F_A} X z \;\simeq\; B F_B$ itself. In this section we will see why
the distinction is important. To understand this, we first define the
following two variants of the Nielsen--Schreier theorem.

\begin{defi}
  We say the \emph{untruncated Nielsen--Schreier theorem} holds if for
  each set $A$ and each subgroup $H \hookrightarrow F_A$ we can choose
  a subset $C_{A, H} \hookrightarrow H$ that freely generates $H$,
  i.e. the lift $F_{C_{A, H}} \to H$ given by universal property is
  an isomorphism.
\end{defi}

\begin{defi}
  We say the \emph{equivariant Nielsen--Schreier theorem} holds if for
  each set $A$ and each subgroup $H \hookrightarrow F_A$ we can choose
  a subset $C_{A, H} \hookrightarrow H$ that freely generates $H$
  satisfying the following condition. Given any equivalence
  $\pi : A \to A'$ we write $\tilde{\pi}$ for the lift to an
  isomorphism $F_A \cong F_{A'}$ given by universal property. We
  require that $C_{A', \tilde{\pi}(H)} = \tilde{\pi}(C_{A, H})$.
\end{defi}

\begin{lem}
  \label{lem:untruncimpliesequivariant}
  The untruncated Nielsen--Schreier theorem implies the equivariant
  Nielsen--Schreier theorem.

  Moreover, the untruncated Nielsen--Schreier theorem restricted to
  free groups generated by merely finite sets implies the equivariant
  Nielsen--Schreier theorem with the same restriction.
\end{lem}

\begin{proof}
  Given sets $A, A'$ and a path $p : A = A'$, write $\coe(p)$ for the
  equivalence $A \to A'$ given by transport. It is straightforward to
  show by path induction that
  $C_{A', \widetilde{\coe(p)}(H)} = \widetilde{\coe(p)}(C_{A,
    H})$. However, by univalence and the fact that the projection from
  $\hset$ to $\types$ is an embedding, for every equivalence
  $\pi : A \simeq A'$, there is a unique path $p : A = A'$ such that
  $\pi = \coe(p)$, so this is true for all equivalences.

  Note that the projection from merely finite sets to sets is an
  embedding, so it remains true for merely finite $A, A'$ that any
  $\pi : A \simeq A'$ can be written uniquely as $\coe(p)$ for some $p
  : A = A'$.
\end{proof}

At this point it is possible to apply Lemma~\ref{lem:lauchlilemma} to
show the equivariant Nielsen--Schreier theorem is false. However, for
the theorem below we will use an alternative proof, for two
reasons. Firstly, we will give a stronger result that the equivariant
version of the theorem fails already for the case of finite index
subgroups of finitely generated free groups, whereas in L\"{a}uchli's
example the free group is infinitely generated, and we can see by
Theorem~\ref{thm:nsfiniteindex} the given subgroup must have infinite
index. Secondly, we can use an easier, more straightforward argument,
since to find a counterexample to equivariance, we only need to show
any set of generators is not fixed by some transposition, rather than
the stronger result that any generating set is not fixed by any
non-trivial transposition. In fact we will see that one of the very simplest
instances of the Nielsen--Schreier theorem suffices to find a
counterexample.

\begin{thm}
  \label{thm:eqnsfalse}
  The equivariant Nielsen--Schreier theorem is false in homotopy type
  theory, extensional type theory and classical mathematics.
  Moreover, it remains false if we restrict to finite index
  subgroups of finitely generated free groups.
\end{thm}

\begin{proof}
  We will given an example of a finite index subgroup of a finitely
  generated free group $F_A$ that does not satisfy the equivariance
  condition for all automorphisms $\pi : A \simeq A$.
  
  We take $A$ to be a type with exactly two elements that we write as
  $a$ and $b$. Note that we have a unique homomorphism
  $\theta : F_A \to \ZZ / 2\ZZ$ that sends both $a$ and $b$ to $1$. We
  take $H \leq F_A$ to be the kernel of $\theta$. We first observe
  that $H$ has index $2$, since for homomorphisms $\theta$ in general,
  the index of $\ker(\theta)$ is equal to the size of the image of
  $\theta$, and in this case $\theta$ is surjective and
  the codomain has size $2$.

  Take $\pi$ to be the transposition $(a \; b)$. By the assumption of
  equivariant Nielsen--Schreier we have $C_{A, H}$ such that
  $\tilde{\pi}(C_{A, H}) = C_{A, \tilde{\pi}(H)}$. However, note that
  $\theta(\tilde{\pi}(g)) = \theta(g)$ for all $g \in F_A$, since this is
  clearly true whenever $g \in A$. Hence
  $\tilde{\pi}(H) = H$. It follows that
  $\tilde{\pi}(C_{A, H}) = C_{A, H}$.

  We next show that $C_{A, H}$ has exactly $3$ elements. Certainly $H$
  can be freely generated by a (non-equivariant) set of $3$ elements
  by the Nielsen--Schreier theorem, say $g_1, g_2, g_3$.\footnote{In
    fact we can read off from the proof an explicit set
    $\{a^2, ab, ab^{-1}\}$.}  We only need finitely many elements of
  $C_{A, H}$ to generate each of $g_1, g_2, g_3$, and so a finitely
  enumerable subset\footnote{That is, the surjective image of an
    initial segment of $\mathbb{N}$.} of $C_{A, H}$
  generates all of $H$. By freeness and the fact that $F_A$ has
  decidable equality, it follows $C_{A, H}$ itself is finitely
  enumerable. Again using that $F_A$ has decidable equality it follows
  that $C_{A, H}$ is finite, i.e. in bijection with an initial
  segment of $\mathbb{N}$. One can then show $C_{A, H}$ has
  the same size as any other finite set that freely generates $H$
  using standard arguments from algebra such as \cite[Proposition
  5.75]{rotmanama}, which is constructively valid as stated. Hence
  $|C_{A, H}| = 3$.

  We can now deduce that $\tilde{\pi}$ permutes the set with three
  elements $C_{A, H}$. However, it has order $2$, and any permutation
  of $3$ elements of order $2$ has a fixed point. We deduce that
  $\tilde{\pi}(h) = h$ for some $h \in C_{A, H}$. But, using the
  explicit description of $F_A$ in terms of reduced words, the only
  element of $F_A$ fixed by $\tilde{\pi}$ is the identity, which
  cannot appear in any free generating set, giving a contradiction.
\end{proof}

\begin{cor}
  The untruncated Nielsen--Schreier theorem is false in homotopy type
  theory. Moreover, it remains false if we restrict to finite index
  subgroups of finitely generated free groups.
\end{cor}

\begin{proof}
  By Lemma~\ref{lem:untruncimpliesequivariant} it suffices to show the
  equivariant Nielsen--Schreier theorem in finite index subgroups of
  finitely generated free groups is false, which was Theorem
  ~\ref{thm:eqnsfalse}.
\end{proof}

\section{Conclusion}
\label{sec:conclusion}

We have given a proof of two versions of the Nielsen--Schreier theorem
in homotopy type theory. The proof of the finite index version,
Theorem~\ref{thm:nsfiniteindex} has been verified electronically in
the Agda proof assistant using the HoTT Agda library
\cite{hottagda}. Since we only used axioms that are always present in
homotopy type theory, the formal proof holds in a wide variety of
models. Alternatively, we could have used the new cubical mode now
available in Agda \cite{cubicalagda}. In this case the class of models
is currently limited to those following the Orton-Pitts approach
\cite{pittsortoncubtopos}, but some of the proofs would have been a
bit easier (see below).

In the proof we made good use of concepts in HoTT to transfer ideas
from the topological proof into type theory. In particular we used the
coequalizer higher inductive type to represent the geometric
realisation of graphs and (in particular) free groups. This approach
was even useful in the relatively concrete construction of spanning
trees in Section~\ref{sec:some-basic-graph}. We defined trees and
connected graphs by applying the existing concepts of contractibility
and connectedness of types to the coequalizer of the graph. We were
able to work directly with these definitions throughout, without ever
needing to define or use the more usual notion of path in a graph as a
finite sequence of edges, which was an advantage in the electronic
formalisation.

One minor difficulty in the formal proof is the lemmas on coequalizers
appearing in Section~\ref{sec:hits}, in particular the higher paths
that appear when constructing the equality part of an equivalence by
coequalizer induction. However, many of these difficulties could have been
eliminated by instead using cubical mode. This allows one
to define coequalizers using the usual Agda data syntax, to use
pattern matching instead of elimination terms and it makes
$\beta$-reduction for path constructors definitional, the latter
playing a useful role in these proofs particularly. To demonstrate
this a formalisation of some of the results of Section~\ref{sec:hits} using
cubical mode has been added in a separate directory.

The independence result in Section~\ref{sec:boolean-infty-topos} shows
that the axiom of choice is necessary for the main theorem. It also
demonstrates that independence results in HoTT can now be relatively
straightforward thanks to Shulman's interpretation of HoTT in a
Grothendieck $\infty$-topos \cite{shulmaninftytopunivalence}, Lurie's
construction of enveloping $\infty$-toposes \cite{luriehtt} and the
existing body of work on Grothendieck toposes, in this case the use of
nominal sets to provide simpler categorical versions of proofs using
Fraenkel-Mostowski models.

Finally the result in Section~\ref{sec:untr-niels-schr} provides a
simple example of an important concept in HoTT and its relevance to the
Nielsen--Schreier theorem: mathematics in the presence of univalence is
``inherently equivariant.''

\bibliographystyle{alphaurl}
\bibliography{mybib}{}

\end{document}